\documentclass[preprint]{elsarticle}

\usepackage{mathtools}
\usepackage[font=small,labelfont=bf]{caption}
\usepackage[all]{xy}
\usepackage{ifthen}

\journal{Discrete Mathematics}
\usepackage{amsthm}
\usepackage{amsmath}
\usepackage{amssymb}
\usepackage{xspace}
\usepackage{mathtools}
\usepackage[T1]{fontenc}
\usepackage{tikz-cd}
\usepackage{bbold}
\usepackage{cleveref}

\DeclareMathAlphabet{\mathsuet} {T1} {wesu}{bx}{sl}

\newcommand{\Fraisse}{Fra\"{i}ss\'{e}\xspace} 
\newcommand{\Jarik}{Ne\v set\v ril\xspace}
\newcommand{\aut}[1]{\mathrm{Aut}(#1)}
\newcommand{\age}[1]{\mathrm{Age}(#1)}
\newcommand{\bim}[1]{\mathrm{Bi}(#1)} 
\newcommand{\trg}{{\normalfont(}$\bigtriangleup${\normalfont )}\xspace}

\newcommand{\ctrg}{{\normalfont(}$\therefore${\normalfont)}\xspace}
\newcommand{\img}[1]{\mathrm{im}(#1)\xspace}	
\newcommand{\dom}[1]{\mathrm{dom}(#1)\xspace}	
\newcommand{\repr}[4]{\mathrm{#1(}#2\mathrm{)}_{#3}^{#4}}

\newcommand{\mpl}[1]{
   \ifthenelse{ \equal{#1}{} }
      {\ensuremath{\;\mathrm{M^+}}}
      {\ensuremath{\;\mathrm{M^+}(#1)}}
}
\newcommand{\mmn}[1]{
   \ifthenelse{ \equal{#1}{} }
      {\ensuremath{\;\mathrm{M^-}}}
      {\ensuremath{\;\mathrm{M^-}(#1)}}
}
\newcommand{\oo}[1]{{\sigma}(#1)}

\newcommand{\isotp}[1]{\mathfrak{#1}}

\newtheorem{theorem}{Theorem}
\newtheorem{lemma}[theorem]{Lemma}
\newtheorem{corollary}[theorem]{Corollary}
\newtheorem{proposition}[theorem]{Proposition}
\newtheorem{definition}[theorem]{Definition}

\newdefinition{construction}{Construction}
\newdefinition{remark}{Remark}
\newdefinition{example}{Example}
\newdefinition{notation}{Notation}
\newtheorem{question}[theorem]{Question}
\newtheorem{observation}[theorem]{Observation}
\newtheorem{claim}[theorem]{Claim}
\newtheorem{fact}[theorem]{Fact}

\begin{document}
\begin{frontmatter}
\title{IB-homogeneous graphs}

\author[ia]{Andr\'{e}s Aranda}
\address[ia]{Katedra Aplikovan\'e Matematiky, Univerzita Karlova. Malostransk\'e n\'am\v{e}st\'i 25, Praha 1.}
\ead{andres.aranda@gmail.com}
\begin{abstract}
The Lachlan-Woodrow Theorem identifies ultrahomogeneous graphs up to isomorphism. Recently, the present author and D. Hartman classified MB-homo\-geneous graphs up to bimorphism-equivalence. We extend those results in this paper, showing that every IB-homogeneous graph is either ultrahomogeneous or MB-homogeneous, thus classifying IB-homogeneous graphs.
\end{abstract}


\end{frontmatter}
\section{Introduction}
A relational structure $M$ is \emph{ultrahomogeneous} if every isomorphism between finite substructures of $M$ is a restriction of an automorphism of $M$. Classic examples of ultrahomogeneous structures include an unstructured set, the ordered rational numbers, the universal homogeneous partial order, and the Rado graph. Ultrahomogeneity has been studied heavily due to its connections to group theory (via the study of oligomorphic permutation groups and extremely amenable groups), model theory (elimination of quantifiers, $\omega$-categoricity, and related constructions such as Hrushovski's), combinatorics (structural Ramsey theory), and even topological dynamics (the Kechris-Pestov-Todor\v{c}evi\'c connection).  

\Fraisse's theorem \cite{Fraisse:1953} establishes a correspondence between hereditary classes of finite (or finitely generated if function symbols are allowed) structures with only countably many isomorphism types, the Joint Embedding Property, and the Amalgamation Property on one hand, and ultrahomogeneous structures on the other. Since the theories of ultrahomogeneous structures with finitely many relations of each finite arity are omega-categorical (for an excellent overview of the area, see \cite{MACPHERSON20111599}), the focus is typically on the unique countable model, known as the \Fraisse limit of the class. 

Ultrahomogeneous structures often play a central role in different areas of mathematics. For example, the Rado graph is a model of the almost-sure theory of finite graphs and almost all (in the sense of either measure or Baire category) countable labelled graphs are isomorphic to it (for this and many other properties of the Rado graph, see \cite{Cameron2013}), Hall's universal group is the \Fraisse limit of the class of all finite groups, and Urysohn's universal metric space is the completion of the \emph{rational} Urysohn space, a \Fraisse limit. A more recent example is the construction of \Fraisse Fr\'echet spaces by Kawach and L\'opez-Abad \cite{kawach2021fraisse}.

Ultrahomogeneity is therefore a strong condition with many desirable consequences. Thus, we would like to know as much as possible about the ultrahomogeneous members of each class of structures. This idea led to classification projects including the classification of binary stable ultrahomogeneous structures (see \cite{lachlan1986binary} and \cite{lachlan1986binary2}), the classification of ultrahomogeneous graphs \cite{LachlanWoodrow:1980}, the classification of ultrahomogeneous partial orders \cite{schmerl1979countable}, and the classification of ultrahomogeneous directed graphs \cite{cherlin1998classification}, among others. 

The notion of homomorphism-homogeneity was introduced by Cameron and \Jarik in \cite{CameronNesetril:2006} as a variation on ultrahomogeneity in which homomorphisms between finite substructures of $M$ are restrictions of endomorphisms of $M$.

Later, Lockett and Truss \cite{LockettTruss:2014} introduced finer distinctions in the class of homo\-morphism-homogeneous $L$-structures, characterized by the type of homomorphism between finite induced substructures of $M$ and the type of endomorphism to which such homomorphisms can be extended. In total, they introduced 18 \emph{morphism-extension classes}, partially ordered by inclusion. 

We call a relational structure $M$ XY-homogeneous if every X-morphism between finite induced substructures of $M$ extends to a Y-morphism $M\to M$, where $\mathrm{X\in\{I,M,H\}}$ and $\mathrm{Y\in\{H,I,A,E,B,M\}}$. The meaning of these symbols is as follows:
\begin{itemize}
\setlength\itemsep{0em}
\item[$\ast$]{H: homomorphism.}
\item[$\ast$]{M: monomorphism (injective homomorphism).}
\item[$\ast$]{I: isomorphism; an isomorphism $M\to M$ is also called a self-embedding.}
\item[$\ast$]{A: automorphism (surjective isomorphism $M\to M$).}
\item[$\ast$]{E: epimorphism (surjective homomorphism $M\to M$).}
\item[$\ast$]{B: bimorphism (surjective monomorphism $M\to M$).}
\end{itemize}

For example, ultrahomogeneous structures are IA-homogeneous structures in this formulation, and the homomorphism-homogeneous structures of Cameron and \Jarik are our HH-homogeneous structures. The partial order of morphism-extension classes of a general class of countable relational structures is presented in Figure \ref{fig:ctblestrs}.
\begin{figure}[h!]
\centering
\includegraphics[scale=0.9]{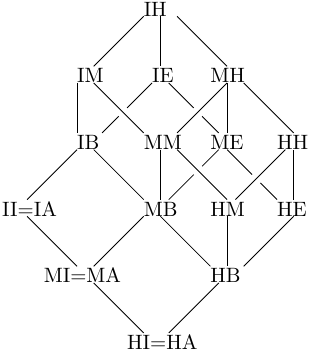}
\caption{Morphism-extension classes of countable structures, partially ordered by $\subseteq$.}
\label{fig:ctblestrs}
\end{figure}

Each of these morphism-extension classes poses the question, to what extent is it possible to replicate the theory of ultrahomogeneous structures in the XY-homogeneous setting? Indeed, much of the study of homomorphism-homogeneity has to do with finding parallels for the central results in the study of ultrahomogeneous structures. Examples of this trend include:
\begin{enumerate}
\item{The quest for \Fraisse theorems for all the morphism-extension classes. The first analogues of the analogues of \Fraisse's theorem appeared in \cite{CameronNesetril:2006} and \cite{Pech2016towards}, and a more uniform approach that could be used to find \Fraisse theorems in most of the morphism-extension classes was identified by Coleman in \cite{Coleman:2018};}
\item{results about the oligomorphic automorphism groups of ultrahomogeneous structures inspired the study of oligomorphic endomorphism monoids of homomorphism-homo\-geneous structures, see for example in \cite{mavsulovic2011oligomorphic} and \cite{ColemanEvansGray:2019};}
\item{the Ryll-Nardzewski theorem and elimination of quantifiers led to attempts to find analogues in the wider class of homomorphism-homo\-geneous structures, see \cite{Pech2016towards}; and }
\item{classification projects, such as the classification of homomorphism-homo\-geneous partial orders \cite{LockettTruss:2014}, finite homomorphism-homogeneous graphs and tournaments with loops \cite{ilic2012homomorphism}, \cite{IlicMasulivicRajkovic2008}, and tournaments \cite{feller2020classification}, among others.}
\end{enumerate}

The present paper follows the same trend, and started as a case study to determine the analogue of \Fraisse's theorem for IB-homogeneous structures. In addition to yielding the IB-analogue of the Lachlan-Woodrow theorem, which we present here, the ideas from Section \ref{sec:represented} have been adapted to the general relational case and helped us to identify and prove analogues of \Fraisse's theorem for IB- and IM-homogeneous structures (\cite{arandaimfr}, in preparation). 

Our focus in this paper is on countably infinite IB-homogeneous simple undirected graphs, that is, countable graphs $G$ for which every isomorphism between finite induced subgraphs is a restriction of a bijective endomorphism $G\to G$ (a \emph{bimorphism} of $G$). 

So far a full classification of countable homomorphism-homogeneous graphs has eluded us. The present paper fills one of the many gaps in that classification project. At this point, we know the partial order of morphism-extension classes of countable graphs and countable connected graphs (Figure \ref{fig:extclassesgraphs}, proof in \cite{aranda2020poset}), as well as classifications of countable graphs for HI (only $K_\omega$), IA (the Lachlan-Woodrow theorem, included below as Theorem \ref{thm:lw} \cite{LachlanWoodrow:1980}), MI ($K_\omega$ and its complement), and MB (see \cite{ARANDA2020103063}). In this paper, we extend this partial classification to include IB-homogeneous graphs (Theorem \ref{thm:ibclass}).

The general structure of the proof is as follows: in Section \ref{sec:represented} we introduce \emph{represented monomorphisms} and show that in an IB-homogeneous structure $M$ these form the largest family of monomorphisms between finite substructures that are restrictions of bimorphisms of $M$. In Section \ref{sec:graphs} we apply what we learned in Section \ref{sec:represented} to graphs, and prove that an IB-homogeneous graph that represents the monomorphism mapping a nonedge to an edge is MB-homogeneous. The other option, where an IB-homogeneous graph does not represent that monomorphism, implies ultrahomogeneity. Since we have classifications for MB-homogeneous graphs and ultrahomogeneous graphs, this completes the classification of MB-homogeneous graphs. 

\begin{figure}[h!]
\centering
\includegraphics[scale=0.9]{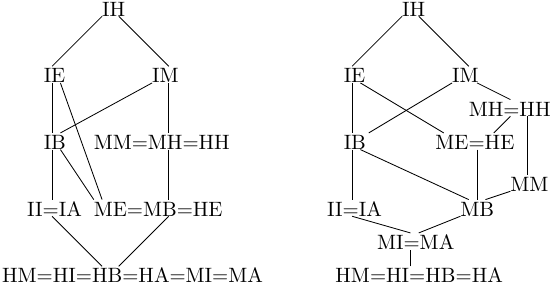}
\caption{Morphism-extension classes of countable connected graphs (left) and countable graphs, partially ordered by $\subseteq$.}
\label{fig:extclassesgraphs}
\end{figure}

\section{Represented Monomorphisms}\label{sec:represented}
In this section we will give an alternative characterization of IB-homogeneity in terms of a family of monomorphisms with finite domain which we call \emph{represented monomorphisms} (the formal definition will be given later).

We will need some notation and conventions. First, we clarify that all graphs in this paper are countable, undirected, and loopless. The age of a relational $L$-structure $M$ is usually defined as the class $\age{M}$ of all finite $L$-structures that embed into $M$. In this work, we will think of the age of $M$ as containing only one representative from each isomorphism type of finite structures embeddable into $M$. There is no real loss in this approach, and it has the salubrious effect of transforming statements about the proper class $\age{M}$ into statements about a countable set. The unique element of the age isomorphic to a finite induced substructure of $M$ will be denoted by the same letter as the substructure, but in Gothic typeface.

\begin{notation}~
\begin{enumerate}
\item{We will use $A\sqsubset M$ to indicate that $A$ is a finite subset of $M$. We identify $A\sqsubset M$ with the substructure induced on it.}
\item{The edge relation in a graph will be denoted by $\sim$. All graphs in this paper are simple and undirected.}
\item{A \emph{nonedge} in a graph $G$ is a pair of distinct vertices $x,y\in G$ with $x\not\sim y$.}
\item{We will denote the restriction of a function $F$ to a subset $X$ of its domain by $F_X$, or $F\upharpoonright X$ when other subindices are present.}
\item{The left inverse of an injective function $g$ will be denoted by $\overline{g}$. We reserve the notation $g^{-1}$ for two-sided inverses.}
\item{If $A\sqsubset M$, then $\isotp{A}$ is the unique element of $\age{M}$ isomorphic to $A$.}
\item{$\bim{M}$ is the bimorphism monoid of $M$.}
\end{enumerate}
\end{notation}

We reserve the name \emph{bimorphism} for bijective homomorphisms in which the domain and image are the same graph. When this condition is not met, we speak of \emph{bijective homomorphisms}.

\begin{example}\label{ex:1}~
\begin{enumerate}
\item{If $G$ is a finite graph, then $\bim{G}=\aut{G}$. This follows easily by counting edges; from this it follows that all the finite IB-homogeneous graphs are ultrahomogeneous. The finite ultrahomogeneous graphs were classified by Gardiner \cite{Gardiner:1976}.}
\item{The Rado graph is the unique countable graph $R$ with the following property: for all finite disjoint sets of vertices $A,B$ there exists some vertex $x$ such that $x\sim a$ for all $a\in A$ and $x\not\sim b$ for all $b\in B$. It easy to see that the isomorphism type of $R$ does not change if we add any finite set of edges to it. This means that there exist enumerations $R=\{a_i:i\in\omega\}=\{b_i:i\in\omega\}$ in which $a_i\sim a_j$ implies $b_i\sim b_j$ for all $i,j\in\omega$, but there are finitely many pairs $i,j$ for which $a_i\not\sim a_j$ and $b_i\sim b_j$. In this case the bijective function $f\colon a_i\mapsto b_i$ is is an example of a bimorphism $R\to R$ that is not an automorphism.}
\item{Consider the graphs $$G_0\coloneqq (\{a_0,b_0\};\varnothing), G_1\coloneqq (\{a_1,b_1\};\{(a_1,b_1),(b_1,a_1)\}).$$Then the function $m\colon G_0\to G_1$ mapping $a_0$ to $a_1$ and $b_0$ to $b_1$ is a bijective homomorphism. This particular homomorphism will play an important role in section \ref{sec:graphs}.}
\end{enumerate}
\end{example}

\begin{notation}
We adapt the notation from the last item of Example \ref{ex:1} to our convention of using Gothic letters for homomorphisms between elements of the age of a graph, and reserve the symbol $\isotp{m}$ to denote a bijective graph homomorphism mapping a nonedge to an edge.
\end{notation}

When compared with other morphism-extension classes, the six classes at the top of the hierarchy in Figure \ref{fig:ctblestrs} (IH, IM, IE, MH, IB, ME) have a mismatch between the type of local homomorphism and the type of endomorphism. To see what we mean by this, note that the finite restrictions of an endomorphism are homomorphisms; likewise, the finite restrictions of a injective endomorphism or a bimorphism will be monomorphisms. The mismatch lies in the fact that when XY$\in\{\mathrm{IH, IM, IE, MH, IB, ME}\}$, then the restrictions of an endomorphism of type Y is not necessarily in the class of local homomorphisms X. 

What the mismatch tells us is that we should not focus exclusively on local X-morphisms, but accept a larger class of local homomorphisms. We call these homomorphisms \emph{represented}, and define them formally below.

Given any $G$, we will think of $\age{G}$ as the set of objects of a category. The morphisms of this category are the total monomorphisms between elements of $\age{G}$. The vertex set of any $\isotp{A}\in\age{G}$ is disjoint from $G$, and we think of the arrows of $\age{G}$ as archetypes of local homomorphisms in $G$. We reflect this convention in our notation by using Gothic typeface for elements of the age and arrows between them.

\begin{definition}\label{def:repr}
Let $M$ be a relational structure, $\isotp{A,B}\in\age{M}$, and $e_A\colon\isotp{A}\to M, e_B\colon\isotp{B}\to M$ be embeddings with images $A,B$ respectively.
\begin{enumerate}
\item{A monomorphism $\isotp{f}\colon\isotp{A}\to\isotp{B}$ is \emph{manifested by} $f\colon A\to B$ \emph{over} $e_A,e_B$, if $f\circ e_A=e_B\circ \isotp{f}$. We also say that $f$ is a \emph{manifestation} of $\isotp{f}$ in this situation.}
\item{A monomorphism $\isotp{f}\colon\isotp{A}\to\isotp{B}$ is \emph{represented in} $\bim{M}$ \emph{over} $e_A,e_B$ if there exists a bimorphism $F\in\bim{M}$ such that $\isotp{f}$ is manifested by $F_A\colon A\to B$. We will also say that $F$ \emph{represents} $\isotp{f}$ \emph{over} $e_A,e_B$ in this situation.}
\item{Given two relational structures $A$ and $B$, $\repr{Mon}{A,B}{}{}$ denotes the set of all mono\-morphisms $A\to B$; similarly, $\repr{Emb}{A,B}{}{}$ denotes all embeddings $A\to B$.}
\item{Given $\isotp{A,B}\in\age{M}$, $e\in\repr{Emb}{\isotp{A},M}{}{},$ and $e'\in\repr{Emb}{\isotp{B},M}{}{}$, we use $\repr{Mon}{e,e'}{M}{Bi}$ to denote the set $$\{\isotp{f}\in\repr{Mon}{\isotp{A,B}}{}{}:\exists F\in\bim{M}(F_{\img{e}}\circ e=e'\circ\isotp{f})\},$$ that is, the set of monomorphisms from $\isotp{A}$ to $\isotp{B}$ represented in $\bim{M}$ over $e,e'$. We use $\repr{Mon}{\isotp{A}}{M}{Bi}$ for $$\bigcup\{\repr{Mon}{\isotp{A,B}}{M}{Bi}:\isotp{B}\in\age{M}\}.$$}
\end{enumerate}
\end{definition}

\begin{example}\label{ex:2}~
\begin{enumerate}
\item{In some cases, the bimorphism monoid of a graph equals its automorphism group. This is true of all finite graphs (Example \ref{ex:1}, item 1), and also of some infinite graphs. One of them is the universal homogeneous triangle-free graph $H_3$, which is the \Fraisse limit of all finite triangle-free graphs. Among countable graphs, $H_3$ is characterised by not embedding $K_3$ and the property that for all finite disjoint $A,B$ such that $A$ spans no edges, there exists a vertex $x\in H_3$ such that $x\sim a$ for all $a\in A$ and $x\not \sim b$ for all $b\in B$. To see that all bimorphisms of $H_3$ are automorphisms, suppose for a contradiction that $H_3$ had a bimorphism $B\in \bim{H_3}\setminus \aut{H_3}$, then $B$ would map some nonedge $a\not\sim b$ to an edge $B(a)\sim B(b)$, but by the axioms of $H_3$ we know that $a$ and $b$ have a common neighbour $c$, and by virtue of being a homomorphism $B$ maps $c$ to a common neighbour of $B(a)$ and $B(b)$, which is impossible since $H_3$ does not embed a triangle. In the case of $H_3$, the only monomorphisms represented in $\bim{H_3}$ are embeddings.}
\item{On the other hand, we know that the Rado graph has bimorphisms outside its automorphism group (Example \ref{ex:1}, item 2). It follows that $\isotp{m}$ is represented in $\bim{R}$ over some embeddings of a nonedge and an edge into $R$ (IB-homogeneity of $R$ and Proposition \ref{prop:homogeneity} below prove that $\isotp{m}$ is in fact represented in $\bim{R}$ over \emph{all} such pairs of embeddings).}
\item{It is proved in Example 3.11 in \cite{ColemanEvansGray:2019} that for any countable binary sequence $s\colon\omega\to 2$ starting with 0 and containing infinitely many 0's and 1's, the graph on $\omega$ with edge set $\{(i,j):\max\{i,j\}=1\}$ is MB-homogeneous. In particular, $\isotp{m}$ is represented in the bimorphism monoid of any such graph. In Example \ref{ex:3}, we use a graph of this type to produce a more complicated IB-homogeneous structure in which not all monomorphisms are represented.}
\end{enumerate}
\end{example}

Example \ref{ex:2} hints at a dichotomy: the bimorphism monoid of an IB-homo\-geneous graph either represents $\isotp{m}$ and the graph is MB-homogeneous, or does not represent $\isotp{m}$ and the graph is ultrahomogeneous. Section \ref{sec:graphs} contains a proof of this fact.

\begin{observation}
Let $M$ be an IB-homogeneous structure and suppose that $\isotp{f}\colon\isotp{A}\to\isotp{B}$ is represented in $\bim{M}$ over embeddings $e_A,e_B$ with images $A,B\sqsubset M$. Then for any pair of embeddings $e'_A\colon\isotp{A}\to A$ and $e'_B\colon\isotp{B}\to B$, the monomorphism $\isotp{f}$ is represented in $\bim{M}$ over $e'_A,e'_B$.
\end{observation}
\begin{proof}
Since the embeddings $e_A$ and $e'_A$ have the same $A\sqsubset M$ as image (and similarly for $e_B,e_B'$), there exist local isomorphisms $s\colon A\to A$ and $t\colon B\to B$ such that $s\circ e_A'=e_A$ and $t\circ e_B=e_B'$. 

Let $S$, $T$, and $F$ be bimorphisms of $M$ extending $s$, $t$ and $f$, respectively. Now,
\[
\begin{split}
(T\circ F\circ S)_A\circ e_A'&=T_B\circ F_A\circ (S_A\circ e_A')=t\circ (f\circ e_A)=\\
&=t\circ e_B\circ\isotp{f}=e'_B\circ\isotp{f}.
\end{split}
\]
This proves that $T\circ F\circ S\in\bim{M}$ represents $\isotp{f}$ over $e_A', e_B'$.
\end{proof}

As a consequence, whenever the ambient structure $M$ is IB-homogeneous, we will think of $e_A$ as an equivalence class of embeddings $\isotp{A}\to M$ with image $A$ (under the equivalence relation $e\approx e'$ if $\dom{e}=\dom{e'}$ and there exists some $\sigma\in\aut{\dom{e}}$ with $e'=e\circ\sigma$), rather than a single embedding.

We can use these notions to give an alternative definition of IB-homogeneity.

\begin{proposition}
A relational structure $M$ is IB-homogeneous if and only if for all $\isotp{A}\in\age{M}$, and all embeddings $e,e'\colon\isotp{A}\to M$, $\aut{\isotp{A}}\subseteq\repr{Mon}{e,e'}{M}{Bi}$. 
\end{proposition}
\begin{proof}
Suppose $M$ is IB-homogeneous, and let $\isotp{s}\colon\isotp{A}\to\isotp{A}$ be an automorphism. Take any embeddings $e,e'\colon\isotp{A}\to M,$ and let $A,A'$ be their images. Then $i\coloneqq e'\circ\isotp{s}\circ\overline e\colon A\to A'$ is a local isomorphism, which by IB-homogeneity is a restriction of some bimorphism $I$. It follows that $I_A\circ e=e'\circ\isotp{s}$ and so $\isotp{s}\in\repr{Mon}{e,e'}{M}{Bi}$.

Now suppose that the condition from the statement is satisfied, and let $i\colon A\to A'$ be a local isomorphism. Let $\isotp{A}$ be the element of $\age{M}$ isomorphic to $A$ and $A'$, and choose isomorphisms $e\colon\isotp{A}\to A$ and $e'\colon\isotp{A}\to A'$ such that $i\circ e=e'$.

The function $\isotp{s}\coloneqq\overline{e'}\circ i\circ e\colon\isotp{A}\to\isotp{A}$ is an automorphism. By hypothesis, there exists $S\in\bim{M}$ that represents $\isotp{s}$ over $e,e'$, that is, $S_A\circ e=e'\circ\isotp{s}$. 

Then $S$ extends $i$, for 
\[
S_A\circ e=e'\circ\isotp{s}=e'\circ\overline{e'}\circ i\circ e=i\circ e.
\]

And the result follows
\end{proof}

\begin{proposition}\label{prop:homogeneity}
Suppose $M$ is an IB-homogeneous relational structure and $\isotp{A,B}\in\age{M}$. Let $i,i'\colon\isotp{A}\to M$ and $e,e'\colon\isotp{B}\to M$ be embeddings. Then $\isotp{f}\in\repr{Mon}{i,e}{M}{Bi}$ if and only if $\isotp{f}\in\repr{Mon}{i',e'}{M}{Bi}$.
\end{proposition}
\begin{proof}
Let $A,B$ be the images of $i,e$ and $A',B'$ be the images of $i',e'$. Suppose $\isotp{f}\in\repr{Mon}{i,e}{M}{Bi}$. Then there exists $F\in\bim{M}$ such that $F_A\circ i=e\circ\isotp{f}$. Since $i,i',e,e'$ are embeddings, there exist isomorphisms $j\colon A'\to A$ and $k\colon B\to B'$ which moreover satisfy $j\circ i'=i$ and $k\circ e=e'$. By IB-homogeneity, $j$ and $k$ are restrictions of bimorphisms $J,K$. 

We claim that $K\circ F\circ J$ represents $\isotp{f}$ over $i',e'$. Note that $(K\circ F\circ J)_A=k\circ F_A\circ j$.
\[
(K\circ F\circ J)_A\circ i'=k\circ F_A\circ j\circ i'=k\circ F_A\circ i=k\circ e \circ\isotp{f}=e'\circ\isotp{f},
\]
and $\isotp{f}$ is represented over $i',e'$. The same proof works in the other direction as well.
\end{proof}

The moral here is that IB-homogeneity implies uniformity of representation of monomorphisms, in the sense that if there exist $e,e'$ such that $\isotp{f}$ is represented over $e,e'$ in $\bim{M}$, then $\isotp{f}$ is represented over \emph{all} pairs of embeddings. In other words, if $M$ is an IB-homogeneous structure, then $\repr{Mon}{e,e'}{M}{Bi}$ depends not on the embeddings $e,e'$, but only on the isomorphism types of their domains. As we shall see below, in the case of a graph $G$ this means that whenever the monomorphism mapping a nonedge to an edge is represented in the bimorphism monoid of an IB-homogeneous graph, then every nonedge in $G$ can be mapped to an edge by some bimorphism of $G$; this, in turn, implies MB-homogeneity (Proposition \ref{lem:infcodep} for connected graphs with connected component), though the proof is not direct.

\section{IB-homogeneous graphs}\label{sec:graphs}
In this section we prove that any countable IB-homogeneous graph is either ultrahomogeneous or MB-homogeneous. It follows from the Lachlan-Woodrow theorem (Theorem \ref{thm:lw} below; see also \cite{LachlanWoodrow:1980}) and the classification of MB-homo\-geneous graphs in \cite{ARANDA2020103063} that all IB-homogeneous graphs are known up to bi\-morphism-equivalence.

The main theorem in this section is Theorem \ref{thm:ibclass}, an analogue of the Lachlan-Woodrow theorem. We start our analysis by establishing that the class of IB-homogeneous graphs is closed under complements (Lemma \ref{prop:complements}) and observing that for IB-homogeneous graphs if either $G$ or its complement is disconnected then $G$ is ultrahomogeneous. This focuses our attention on connected IB-homogeneous graphs with connected component. We then prove a dichotomy in Observation \ref{obs:ultrahom}, namely that either the monomorphism $\isotp{m}$ is represented in the bimorphism monoid of an IB-homogeneous graph, or the IB-homogneoeus graph is ultrahomogeneous. For the final steps, we use the technical Lemma \ref{lem:toclique} to show in Proposition \ref {lem:infcodep} that a non-ultrahomogeneous IB-homogeneous graph is MB-homogeneous; the proof of Proposition \ref {lem:infcodep} uses Fact \ref{fact:coleman}, a sufficient condition for MB-homogeneity in graphs.

The complement of a graph $G=(V,E)$ is $\overline G=(V,[V]^2\setminus E)$. Note that we assume $G$ and $\overline G$ have the same vertex set. 

We remind the reader that all graphs in this paper are countable, and by ``subgraph'' we always mean ``induced subgraph.''

\begin{definition}
Let $G,H$ be graphs. A function $f\colon G\to H$ is an antihomomorphism if $u\not\sim v$ in $G$ implies $f(u)\not\sim f(v)$. A bijective antihomomorphism $G\to G$ is an \emph{antibimorphism} of $G$.
\end{definition}

\begin{observation}
Let $G,H$ be graphs. A function $F\colon G\to H$ is a bijective homomorphism if and only if $F^{-1}$ is a bijective homomorphism $\overline H\to \overline G$.
\end{observation}
\begin{proof}
Suppose that $F$ is a bijective homomorphism. If $x\sim y$ in $\overline H$, then $x\not\sim y$ in $H$. Since $F$ is bijective and preserves edges, but may map a nonedge to any pair of distinct elements, the preimage of a nonedge is always a nonedge, so $F^{-1}(x)\not\sim F^{-1}(y)$ in $G$, or, equivalently, $F^{-1}(x)\sim F^{-1}(y)$ in $\overline{G}$. The same argument proves the other direction.
\end{proof}

Using $(F^{-1})^{-1}=F$ and $\overline{\overline G}=G$, the four conditions below are equivalent.

\begin{corollary}\label{obs:anti}
Let $G,H$ be graphs and suppose that $F:G\to H$ is a function. The following are equivalent:\
\begin{enumerate}
\item{$F$ is a bijective homomorphism $G\to H$,}
\item{$F^{-1}$ is a bijective homomorphism $\overline H\to\overline G$,}
\item{$F$ is a bijective antihomomorphism $\overline G\to\overline H$, and}
\item{$F^{-1}$ is a bijective antihomomorphism $H\to G$.}
\end{enumerate}
\end{corollary}

The easy proposition below will be in the background for the rest of the paper. 

\begin{proposition}\label{prop:leftinv}
If $G$ is an IB-homogeneous graph, then the left inverse of every finite represented monomorphism can be extended to an antibimorphism of $G$.
\end{proposition}
\begin{proof}
Let $f\colon X\to Y$ be a manifestation of a represented $\isotp{f}\colon\isotp{X}\to\isotp{Y}$ for some $\isotp{X},\isotp{Y}\in\age{G}$ . Since $G$ is IB-homogeneous, we know that $f$ is a restriction of some $F\in\bim{G}$. Note that $\overline{f}\colon\img{f}\to X$ is the restriction of $F^{-1}$ to $\img{f}$, so that $F^{-1}$ is an antibimorphism of $G$ (by Corollary \ref{obs:anti}) that extends $\overline{f}$.
\end{proof}

\begin{lemma}\label{prop:complements}
If $G$ is an IB-homogeneous graph, then so is $\overline G$.
\end{lemma}
\begin{proof}
If $f\colon\overline X\to\overline Y$ is a finite isomorphism in $\overline G$, then so is $f^{-1}\colon Y\to X$, where we think of $X$ and $Y$ as embedded in $G$. By IB-homogeneity, $f^{-1}$ is represented in $G$, so by Proposition \ref{prop:leftinv} $f$ is a restriction of an antibimorphism $B\colon G\to G$. Now Corollary \ref{obs:anti} tells us that $B^{-1}$ is a bimorphism $\overline G\to\overline G$. Clearly, $B^{-1}$ extends $f$ and $\overline G$ is IB-homogeneous.
\end{proof}

Recall from \cite{ARANDA2020103063} that in an ambient graph $G$ we call a vertex $v$ a \emph{cone} over $X\subset G$ if $v\sim x$ for all $x\in X$. Similarly, we call $v$ a \emph{co-cone} over $X$ if $v\notin X$ and $v\not\sim x$ for all $x\in X$. 

From this point on, whenever $G$ is an IB-homogeneous graph and $\isotp{f}$ is mono\-morphism between elements of $\age{G}$, we write ``$G$ represents $\isotp{f\ }$'' to mean ``$\isotp{f}$ is represented in $\bim{G}$''.
\begin{observation}\label{obs:ultrahom}
If $G$ is an IB-homogeneous graph that does not represent $\isotp{m}$, then $G$ is ultrahomogeneous.
\end{observation}
\begin{proof}
Every isomorphism between finite substructures extends to a bimorphism, which cannot map any nonedge to an edge, as in that case $\isotp{m}$ would be represented in $\bim{G}$. It follows that under these conditions the extension is always an automorphism.
\end{proof}

\begin{proposition}\label{prop:comprepm}
If $G$ is an IB-homogeneous graph that represents $\isotp{m}$, then $\overline G$ also represents $\isotp{m}$.
\end{proposition}
\begin{proof}
Let $M$ be any bimorphism that represents $\isotp{m}$. Then by Corollary \ref{obs:anti}, the same permutation of vertices is an antibimorphism of $\overline G$ mapping an edge to a nonedge and $M^{-1}$ is a bimorphism of $\overline G$ that represents $\isotp{m}$.
\end{proof}

Our next result proves that in an IB-homogeneous graph $G$ that represents $\isotp{m}$, every monomorphism whose domain is an independent set, or whose image is a clique, is represented in the bimorphism monoid of $G$. 
\begin{lemma}\label{lem:toclique}
If $G$ is an IB-homogeneous graph that represents $\isotp{m}$, then for all $X\sqsubset G$ there exist bimorphisms $C$ and $E$ of $G$ such that the image of $X$ under $C$ is a clique and the preimage of $X$ under $E$ is an independent set. In particular, $G$ embeds arbitrarily large cliques and independent sets.
\end{lemma}
\begin{proof}
Suppose that $G$ is IB-homogeneous and represents $\isotp{m}$, and let $X\sqsubset G$ be any finite subset of $G$. 

If $X$ is a clique, then the identity bimorphism maps $X$ to a clique and we are done. Otherwise, there is a nonedge $x\not\sim y$ in $X$. Let $u\sim v$ be any edge of $G$. Since $\isotp{m}$ is represented in $\bim{G}$, the map $x\mapsto u, y\mapsto v$ is a restriction of a bimorphism $C_1$, by Proposition \ref{prop:homogeneity}. The image of $X$ under $F$ is a set of size $X$ with strictly more edges than $X$. Iterating this procedure we obtain a sequence of bimorphisms $C_1,\ldots, C_k$ such that $C_k\circ\ldots\circ C_1\in\bim{G}$ maps $X$ to a complete graph on $|X|$ vertices. 

By Proposition \ref{prop:comprepm}, $\isotp{m}$ is also represented in $\overline G$, so we can map the substructure $\overline X\sqsubset\overline G$ to a clique in $\overline G$ by a bimorphism $D\colon\overline G\to\overline G$. By Corollary \ref{obs:anti}, $E\coloneqq D^{-1}$ is a bimorphism $G\to G$, which clearly maps an independent set to $X$.
\end{proof}

The rest of our argument to classify IB-homogeneous graphs rests on the classification of MB-homogeneous graphs from \cite{ARANDA2020103063} and the Lachlan-Woodrow theorem. We need the definition of \emph{star number} from \cite{ARANDA2020103063}, as well as Properties \trg and \ctrg from \cite{ColemanEvansGray:2019}.

\begin{definition}
Let $G$ be a graph.
\begin{enumerate}
\item{The \emph{star number} of $G$ is $$\oo{G}=\sup\{n\colon K_{1,n}\in\age{G}\}$$
when that number is finite, and $\infty$ otherwise.}
\item{$G$ has Property \trg if every $X\sqsubset G$ has a cone.}
\item{$G$ has Property \ctrg if $\overline G$ satisfies \trg.}
\end{enumerate}
\end{definition}

\begin{fact}[Proposition 3.6 in \cite{ColemanEvansGray:2019}]\label{fact:coleman}
If a countable graph $G$ satisfies \trg and \ctrg, then $G$ is MB-homogeneous.
\end{fact}

\begin{notation}
If $v$ is a vertex in a graph $G$, then $N(v)$ is the set $\{g\in G: G\models v\sim g\}$; similarly $\overline{N}(v)$ is $\{g\in G:\overline{G}\models g\sim v\}$.
\end{notation}

Given two graphs $G$ and $H$ with disjoint vertex sets, we can form the \emph{graph composite} or \emph{lexicographic product} of $G$ and $H$, denoted by $G[H]$, as follows: the vertex set is $G\times H$ and $(g,h)\sim(g',h')$ if $g\sim g'$ in $G$ or $g=g'$ and $h\sim h'$ in $H$. In $G[H]$, each set of the form $\{g\}\times H$ induces an isomorphic copy of $H$ and for any function $f:G\to H$, the set $\{(g,f(g)):g\in G\}$ with its induced subgraph structure in $G[H]$ is isomorphic to $G$. We will use $I_\kappa$ to denote an independent set of size $\kappa$.

We include the Lachlan-Woodrow theorem for completeness.
\begin{theorem}[Lachlan-Woodrow 1980]\label{thm:lw}
Let $G$ be a countably infinite ultrahomogeneous graph. Then $G$ or its complement is isomorphic to one of the following:
\begin{enumerate}
\item{$K_\omega$,}
\item{$I_\omega[K_n]$ for some $n\in\omega$,}
\item{$I_m[K_\omega]$ for some $m\in\omega+1$,}
\item{the Rado graph $R$, or }
\item{the universal homogeneous $K_n$-free graph, for some $n\in\omega, n\geq 3$.}
\end{enumerate}
\end{theorem}

\begin{proposition}\label{lem:disconnected}
Let $G$ be an IB-homogeneous graph. If $G$ or its complement is disconnected, then $G$ is ultrahomogeneous.
\end{proposition}
\begin{proof}
Suppose that $G$ is IB-homogeneous and disconnected.
\begin{claim}\label{claim:first}
Each connected component of $G$ is a clique.
\end{claim}
\begin{proof}
Suppose for a contradiction that there is a nonedge $x\not\sim y$ within a connected component $C$. Let $f$ be the local isomorphism fixing $x$ and sending $y$ to any $v$ in another connected component $D$. The image of a connected component is connected, so IB-homogeneity yields a contradiction: if $F$ is a global extension of $f$, then $v$ and $x$ are in the same connected component.
\end{proof}
\begin{claim}\label{claim:equipotent}
All connected components of $G$ are equipotent.
\end{claim}
\begin{proof}
Let $C$ and $D$ be two connected components, with $c\in C, d\in D$. The map $c\mapsto d$ extends to an injection $C\to D$, and likewise $d\mapsto c$ witnesses $|D|\leq |C|$. By the Cantor-Bernstein Theorem, $|C|=|D|$.
\end{proof}

A disconnected IB-homogeneous graph is therefore a disjoint union of equipotent cliques. Since we are only considering countable graphs, it follows that $G$ is isomorphic to $I_n[K_\omega]$ with $n\in\omega+1$ or $I_\omega[K_n]$ with $n\geq 2$. All such graphs are ultrahomogeneous, as are their complements, by the Lachlan-Woodrow theo\-rem. 
\end{proof}

We now turn our attention to connected IB-homogeneous graphs with connected complement. Our goal is to prove that any such graph that represent $\isotp{m}$ is MB-homogeneous, and the case analysis is given by the pairs $(\sigma(G),\sigma(\overline{G}))$. First, we eliminate the possibility of having finite star number in both $G$ and $\overline{G}$.

\begin{lemma}\label{lem:infstar}
If $G$ is an infinite graph that embeds arbitrarily large finite cliques and independent sets, then $\infty\in\{\oo{G},\oo{\overline{G}}\}$.
\end{lemma}
\begin{proof}
Suppose for a contradiction that $\oo{G}$ and $\oo{\overline{G}}$ are both finite, and choose an independent subset of $X\sqsubset G$ with $|X|>\oo{G}$, as well as a clique $Y\sqsubset G$ with $|Y|>\oo{\overline{G}}$. Since $G$ embeds arbitrarily large cliques and independent sets, we may assume without loss of generality that $X\cap Y=\varnothing$.

Now for all $v\in G\setminus(X\cup Y)$ there exists $x\in X$ and $y\in Y$ such that $v\not\sim x$ and $v\sim y$. In other words, for all $v\in G\setminus(X\cup Y)$ the sets $N(v)\cap(X\cup Y)$ and $\overline{N}(v)\cap(X\cup Y)$ are proper nonempty subsets of $X\cup Y$.

We define a partition of $G\setminus(X\cup Y)$ into finitely many sets. Given proper nonempty $U\subset X, V\subset Y$, define $$S_{U,V}=\{w\in G\setminus(X\cup Y): \overline{N}(w)\cap X=U \wedge N(w)\cap Y=V\}.$$ By the choice of $X$ and $Y$, every vertex from $G\setminus(X\cup Y)$ is in some $S_{U,V}$, and it is easy to see that $(U,V)\neq(U',V')$ implies $S_{U,V}\cap S_{U',V'}=\varnothing.$

Since $G$ is infinite, there exist proper nonempty $A\subset X$ and $B\subset Y$ such that $S_{A,B}$ is infinite. By Ramsey's theorem, $S_{A,B}$ contains an infinite clique or an infinite independent set.
\begin{enumerate}
\item{If $S_{A,B}$ contains an infinite clique, then any $a\in A$ contradicts $\sigma(\overline{G})<\infty$.}
\item{If $S_Y$ contains an infinite independent set, then any $b\in B$ contradicts $\sigma(G)<\infty$.}
\end{enumerate}

Either way we reach a contradiction. This proves the Lemma.
\end{proof}

The next two lemmas connect the star number of a connected IB-hom\-ogeneous graph with connected component to the properties \trg and \ctrg. 

\begin{lemma}\label{lem:infsigtrg}
If $G$ is an IB-homogeneous graph that represents $\isotp{m}$ and $\oo{G}=\infty$, then $G$ satisfies \trg.
\end{lemma}
\begin{proof}
Consider any $X\sqsubset G$. Since $\oo{G}=\infty$ and $G$ is IB-homogeneous, we know that any finite independent set has a cone. By Lemma \ref{lem:toclique}, $X$ is image of an independent set $I$ of size $|X|$ under a bimorphism $F\colon G\to G$. If $c$ is any cone over $I$, then $F(c)$ is a cone over $X$.
\end{proof}

\begin{lemma}\label{lem:finsigctrg}
If $G$ is a connected IB-homogeneous graph that represents $\isotp{m}$, has connected complement, and $\sigma(\overline{G})<\infty$, then $G$ satisfies \trg.
\end{lemma}
\begin{proof}
The condition $\sigma(\overline{G})<\infty$ means that given any vertex $v\in G$ and clique $C\subset G$, $|\overline{N}(v)\cap C|\leq\sigma(\overline{G})$. 

Consider any finite $X\sqsubset G$. Since $G$ embeds arbitrarily large complete graphs (Lemma \ref{lem:toclique}), there exist a clique $K\sqsubset G$ disjoint from $X$ with more than $\sigma(\overline{G})|X|$ vertices. Note that by the first line of this proof, $$\left|\bigcup\{\overline{N}(x)\cap K:x\in X\}\right|\leq\sigma(\overline{G})|X|,$$ so any element of $K\setminus\bigcup\{\overline{N}(x)\cap K:x\in X\}$ is a cone over $X$.
\end{proof}

The following proposition is the last piece of the puzzle.
\begin{proposition}\label{lem:infcodep}
If $G$ is a connected IB-homogeneous graph that represents $\isotp{m}$ and has connected complement, then $G$ is MB-homogeneous.
\end{proposition}
\begin{proof}
By Lemma \ref{lem:infstar} and the symmetry of the hypotheses (apply Lemma \ref{prop:complements} and Proposition \ref{prop:comprepm}), there are only two cases to consider:
\begin{description}
\item[Case 1]{Suppose that $\sigma(G)=\sigma(\overline{G})=\infty$. Then by Lemma \ref{lem:infsigtrg} both $G$ and $\overline{G}$ satisfy \trg, or , equivalently, $G$ satisfies \trg and \ctrg. By Fact \ref{fact:coleman}, $G$ is MB-homogeneous.}
\item[Case 1]{Now suppose that $\sigma(G)=\infty$ and $\sigma(\overline{G})<\infty$. Then by Lemma \ref{lem:infsigtrg} $G$ satisfies \trg, and by Lemma \ref{lem:finsigctrg} $\overline{G}$ satisfies \trg. Equivalently, $G$ satisfies \trg and \ctrg, and is MB-homogeneous by Fact \ref{fact:coleman}.}
\end{description}
\end{proof}

\begin{theorem}\label{thm:ibclass}
If a countable graph is IB-homogeneous, then it is MB-homo\-geneous or ultrahomogeneous.
\end{theorem}
\begin{proof}
If a countable IB-homogeneous graph $G$ does not represent $\isotp{m}$, or if $G$ or its complement is disconnected, then $G$ is ultrahomogeneous, by Observation \ref{obs:ultrahom} and Proposition \ref{lem:disconnected}. And if both $G$ and $\overline{G}$ are connected and $G$ represents $\isotp{m}$, then $G$ is MB-homogeneous by Proposition \ref{lem:infcodep}.
\end{proof}

Given that we know all MB-homogeneous graphs up to bimorphism-equi\-valence (these are four ultrahomogeneous graphs and all graphs that are bi\-morphism-equivalent to the Rado graph, see \cite{ARANDA2020103063}) and all ultrahomogeneous graphs up to isomorphism, Theorem \ref{thm:ibclass} is a classification of IB-homogeneous graphs. More explicitly,

\begin{corollary}
If a countable graph is IB-homogeneous, then it is either ultrahomogeneous or bimorphism-equivalent to the Rado graph.
\end{corollary}

\section{Closing comments and open problems}
As we mentioned in the introduction, this paper started as a case study to find a \Fraisse theorem for IB-homogeneous structures. The notion of represented monomorphism can be extended to the general relational case, and in a separate paper (\cite{arandaimfr}, in preparation) we prove that it is possible to identify the monomorphisms that will be represented by analysing the age alone, that is, without looking at any homogeneous limit. Those so-called \emph{age-extensible} monomorphisms are the basis for the \Fraisse theorems in IB- and IM-homogeneous structures. 

But the notions and methods from Section \ref{sec:represented} can be pushed even further, to the point where we can identify and prove \Fraisse theorems in a uniform way for all 18 classes of homomorphism-homogeneous structures defined by Lockett and Truss. This, and an additional \Fraisse theorem for polymorphism-homogeneous structures, are the subjects of another paper in preparation (joint work with Maja Pech and Bojana Panti\'c). 

The proof of Theorem \ref{thm:ibclass} presented in this paper is admittedly more intricate than expected, considering how easy it is to prove that an IB-homogeneous graph that does not represent $\isotp{m}$ is ultrahomogeneous. Intuitively, if $\isotp{m}$ is represented and $G$ is IB-homogeneous, then we should be able apply $\isotp{m}$ to add edges one by one (or as few as possible at a time, depending on the contents of $\age{G}$) using to any $X\sqsubset G$, so that any monomorphism can be proved to be represented in $G$. I would be very interested if anyone found a proof of Theorem \ref{thm:ibclass} that follows this approach.   

For general relational structures, it is not true that IB is the union of IA and MB. Our final example is an IB-homogeneous multigraph with coloured edges that is neither ultrahomogeneous nor MB-homogeneous.

\begin{example}\label{ex:3}
Let $H_3=(M;R)$ denote the universal homogeneous triangle-free graph. We will expand $H_3$ by a new edge relation $G$, so that the resulting structure is IB-homogeneous, but neither IA- nor MB-homogeneous. 

Given $X\sqsubset M$ and $v\in M\setminus X$, let $k_{X,v}$ be the function $X\to 2$ defined by $$k_{X,v}(x) = \begin{cases} 0, &\text{if } (M; R)\models\neg R(x,v)\\ 1,&\text{if } (M; R)\models R(x,v).\end{cases}$$

By $\omega$-categoricity of $H_3$, there are only finitely many distinct functions $k_{X,v}$ for each $X\sqsubset M$. Additionally, it follows from the axioms of $H_3$ (see Example \ref{ex:2}, item 1) that for each $X\sqsubset H_3$ and $v\in H_3\setminus X$ there exist infinitely many $v'\in M$ such that $k_{X,v}=k_{X,v'}$.

We need an enumeration $M=\{a_i:i\in\omega\}$ in which for all $X\subset M$ and $v\in M\setminus X$, the sets $\{j\in\omega:j\text{ is odd and } k_{X,a_j}=k_{X,v}\}$ and $\{j\in\omega:j\text{ is even and } k_{X,a_j}=k_{X,v}\}$ are both infinite. Fix a bijection $Q\colon M\to\omega$. We proceed by induction:
\begin{description}
\item[Step 1:]{Let $a_0\coloneqq Q^{-1}(0)$ and  $B_1\coloneqq\{a_0\}$. Over $B_1$ there are only two functions $k_{B_1,v}$, corresponding to neighbours and non-neighbours of $a_0$ in $H_3$; choose the two distinct neighbours of $a_0$ with smallest even and odd value under $Q$ as $a_1, a_2$ and the two distinct non-neighbours of $a_0$ with smallest even and odd value under $Q$ as $a_3$ and $a_4$, and let $B_2 = \{a_0,a_1,a_2,a_3,a_4\}$. }
\item[Step $\mathbf{n+1}$:]{Now suppose that we have succeeded in constructing $B_n$ so that for each function in $\{k_{B_{n-1},v}:v\in H_3\setminus B_n\}$, the set $B_n$ contains elements with odd and even index defining the same function. Enumerate the functions $k_{B_n,v}$ as $k_0, \ldots, k_m$. For each $i\in m$, let $a_{n+2i+1}$ and $a_{n+2i+2}$ be the vertices with smallest even and odd value under $Q$ that have not been enumerated yet and have $k_{B_n,v}=k_i$; define $B_{n+1}$ as the union of $B_n$ and the set of all the vertices chosen in this step. This process extends the partial enumeration and preserves the property at the beginning of this paragraph. }
\end{description}

After $\omega$ steps, we have an enumeration $M = \{a_i:i\in\omega\}$. From this point on, $A_n$ denotes the set $\{a_i:i<n\}$.

\begin{claim}\label{cl:niceenum}
For any $X\sqsubset H_3$ and $v\in H_3\setminus X$, the sets $$\{j\in\omega:j\text{ is odd and } k_{X,a_j}=k_{X,v}\}$$ and $$\{j\in\omega:j\text{ is even and } k_{X,a_j}=k_{X,v}\}$$ are both infinite.
\end{claim}
\begin{proof}
Let $k$ denote any $k_{X,v}$. The set $X$ is contained in $A_{\max\{i\in\omega:a_i\in X\}}$, which is itself contained in some $B_m$. By construction, we know that $B_{m+1}$ contains vertices $w$ with even and odd index such that $k_{B_m,w}\upharpoonright X = k_{X,w}=k_{X,v}$. Suppose for a contradiction that for only finitely many $j\in\omega$, $a_{2j}$ satisfies $k_{X,a_{2j}}=k_{X,v}$.

Let $t$ be the least natural number such that $a_s\in B_t$, where $s=\max\{j\in\omega:k_{X,a_j}=k_{X,v}\}$ . By the axioms of $H_3$, there exists some $y\in H_3\setminus B_t$ such that $$k_{B_t,y}(u)=\begin{cases}1,&\text{if }k_{X,a_s}=1,\\ 0,&\text{otherwise.}\end{cases}$$ By construction, there exists an even $s'>s$ such that $k_{X,a_{s'}}=k_{X,v}.$ This contradicts the definition of $s$, proving the claim (the same argument can be used for odd indices).
\end{proof}

We are now ready to add edges of type $G$. Let $p\colon\omega\to 2$ be the function mapping all even numbers to 0 and all odd numbers to 1. Given $m,n\in\omega$, $G(a_m,a_n)$ and $G(a_n,a_m)$ hold if $m\neq n$ and $p(\max\{m,n\})=1$. 

\begin{claim}
The multigraph $(M; R,G)$ is IB-homogeneous.
\end{claim}
\begin{proof}
Let $e_0\colon X_0\to Y_0$ be an isomorphism, with $X_0,Y_0\sqsubset M$. Suppose that $e_n\colon X_n\to Y_n$ is a monomorphism extending $e_0$ (we allow the case $n=0$). Let $s$ be the least index such that $e_n$ is not defined on $a_s$. Consider the function $f_s\colon Y\to 2$ given by $$f_s(e_n(x))=\begin{cases}0,&\text{if }k_{X_n,a_s}(x)=0\\ 1, &\text{if }k_{X_n,a_s}(x)=1\end{cases}$$ By Claim \ref{cl:niceenum}, there exist infinitely many odd indices $j$ such that $k_{Y,a_j}=f_s$. Let $t$ be the least such index such that $a_t\notin Y$, and $e_{n+1}=e_n\cup\{(a_s,a_t)\}$; let $X_{n+1}=X_n\cup\{a_s\}$ and $Y_{n+1}=Y\cup\{a_t\}$. By the definition of the relation $G$, $e_{n+1}$ is a monomorphism.

We follow the same pattern to find preimages for any vertex outside $Y_n$. If $e_n\colon X_n\to Y_n$ is a monomorphism and $a_u$ is the vertex with least index such that $a_u\notin Y_n$, then we apply Claim \ref{cl:niceenum} to find an even index $v$ such that $a_v\notin X_n$ and whose $R$-neighbours in $X_n$ are the preimages under $e_n$ of the neighbours of $a_u$ in $Y_n$. Extend $e_n$ as $e_{n+1}=e_n\cup\{(a_v,a_u)\}$, and let $X_{n+1}=X_n\cup\{a_v\}, Y_{n+1}=Y_n\cup\{a_u\}$

Now $E=\bigcup\{e_n:n\in\omega\}$ is a bijective endomorphism of $(M; R,G)$ that extends $e_0$.
\end{proof}

Notice that the reduct $(M; E)$ is ultrahomogeneous and $(M;G)$ is MB-homogeneous but not ultrahomogeneous. To prove the last assertion, notice that the graph $(M;G)$ is connected with connected component and embeds an infinite clique, but is not isomorphic to the Rado graph because it is $C_4$-free, so it cannot be ultrahomogeneous by the Lachlan-Woodrow theorem. 

The full structure $(M; R,G)$ is not MB-homogeneous because the monomorphism mapping a nonedge to an edge of type $R$ cannot be extended to a monomorphism, by the same argument from Example \ref{ex:2} item 2 applied to the $R$-edges. 

To see that $(M; R,G)$ is not ultrahomogeneous, notice that there must be pairs of even indices $i>i'$ and $j>j'$ with $j-j'>i-i'$ such that $\neg R(a_i,a_i')$ and $\neg R(a_j,a_j')$ (this follows from $K_3$-freeness of $(M; R)$ and the fact that $\{a_{2j}:j\in\omega\}$ contains infinitely many neighbours of $a_0$). Then the local isomorphism $j\mapsto i, j'\mapsto i'$ cannot be extended to an automorphism of $(M; R,G)$ because the number of vertices with odd index between $i$ and $i'$ is smaller than the nuber of vertices with odd index between $j$ and $j'$: at least one vertex with an odd index $\ell$ between $j$ and $j'$ has to be mapped to some $a_p$, with $p$ odd and greater than $i$, thus introducing new $G$-edges.
\end{example}

Example \ref{ex:3}  shows that superpositions of IA- and MB-homogeneous structures can produce IB-homogeneous structures. The \Fraisse theorem for IB-homogeneous structures allows us to adapt the definition of free superposition of homogeneous structures (see for example \cite{bodirsky2015ramsey}), so that the careful enumeration of the universal homogeneous triangle-free graph in Example \ref{ex:3} becomes unnecessary. This motivates the following two questions.
\begin{question}
Is it true that the free superposition of an ultrahomogeneous structure and an MB-homogeneous structure is an IB-homogeneous structure? (My guess: yes, and probably easy.)
\end{question}

\begin{question}
Conversely, is it true that if $M$ is an IB-homogeneous $L$-structure, then there exist reducts $M_0, M_1$ of $M$ such that $M_0$ is ultrahomogeneous and $M_1$ is MB-homogeneous? (My guess: no.)
\end{question}

Finally, a question about when we can guarantee that IB=IA$\cup$MB.
\begin{question}
Is it true that for all languages $L$ containing only one relation, every countable IB-homogeneous $L$-structure is ultrahomogeneous or MB-homogeneous? If not, then what is the least arity such that a single-relation language has countable IB-homogeneous structures that are neither ultrahomogeneous nor MB-homogeneous?
\end{question}

\section{Acknowledgements}
Many thanks to the anonymous referee for a careful reading of the paper and for their useful suggestions. I also thank Jan Hubi\v{c}ka and David Bradley-Williams, who patiently listened to preliminary and often ranting versions this work.

A good portion of the research in this paper was completed while I was a postdoc at the Institut f\"ur Algebra, Technische Universit\"at Dresden, funded by the ERC under the European Union's Horizon 2020 Research and Innovation Programme (grant agreement No. 681988, CSP-Infinity). Additional research, examples, and editing were completed at KAM, funded by project 18-13685Y of the Czech Science Foundation (GA\v{C}R).

\bibliography{morph}
\end{document}